\documentclass[11pt]{amsart}
\usepackage{geometry}
\usepackage{amsmath, amssymb, amsthm, mathrsfs}
\usepackage{mathtools}
\usepackage[protrusion=true,expansion=false]{microtype}
\usepackage{enumitem}
\usepackage[dvipsnames]{xcolor}
\usepackage{hyperref}
\hypersetup{colorlinks=true, citecolor=red, linkcolor=blue, filecolor=magenta, urlcolor=red}

\usepackage[capitalize,nameinlink]{cleveref}

\newtheorem{thm}{Theorem}[section]

\theoremstyle{definition}

\newtheorem{rem}[thm]{Remark}

\DeclareMathOperator{\Spec}{Spec}
\DeclareMathOperator{\Sing}{Sing}

\title{A question on klt type varieties of Han and Jiang}
\author{Jihao Liu}
\address{Department of Mathematics, Peking University, No. 5 Yiheyuan Road, Haidian District, Beijing 100871, China}
\address{Beijing International Center for Mathematical Research, Peking University, No. 5 Yiheyuan Road, Haidian District, Beijing 100871, China}
\email{liujihao@math.pku.edu.cn}
\date{\today}

\begin{document}

\begin{abstract}
We prove that being of klt type is not an open condition in flat families of varieties. This answers a question of Han and Jiang. The construction in this paper substantially uses generative AI: the general idea for the counterexample was suggested by ChatGPT Pro 5.5, and the explicit example was found and proved by the Rethlas system.
\end{abstract}

\maketitle

\section{Introduction}\label{sec:introduction}

Han and Jiang asked \cite[Problem~4.6]{HJ26} whether being of klt type is an open condition in a flat family of varieties. We answer their question by using Singh's determinantal family \cite{Sin99}.

\begin{thm}\label{thm:main}
Being of klt type is not an open condition in flat families of varieties. As a consequence, \cite[Problem~4.6]{HJ26} has a negative answer.
\end{thm}
\begin{proof}
Let \(R=\mathbb C[A,B,C,D,T]/I_T\), where \(I_T\) is generated by the
\(2\times 2\) minors of
\[
M_T=
\begin{pmatrix}
A^2+T^5 & B & D\\
C & A^2 & B^2-D
\end{pmatrix},
\]
and set
\[
\mathcal R
=
R\otimes_{\mathbb C[T]}\mathbb C[s,E],
\qquad T\mapsto sE.
\]
Equivalently,
\[
\mathcal R\simeq
\mathbb C[s,A,B,C,D,E]/I_{sE}.
\]
Let \(\mathcal{X}:=\Spec\mathcal{R}\), and let
\(\pi\colon \mathcal X\to \mathbb A^1_s\) be the morphism induced by
\(\mathbb C[s]\to \mathcal R\).

We first prove that the special fiber is of klt type. Put \(S=R/TR\). For a graded ring \(G\), write
\[
G^{(5)}=\bigoplus_{q\geq 0}G_{5q}
\]
for the fifth Veronese subring. Let
\[
H_{\mathbb K}
=
\mathbb K[A,X,Y]/(A^2-XY(X^2-Y)),
\]
for any field $\mathbb K$ with grading
\[
\deg A=5,\qquad \deg X=2,\qquad \deg Y=4.
\]
The explicit formulas in the proof of \cite[Proposition~4.3]{Sin99} (stated there over an algebraically closed field $K$ of characteristic $p>2$, but the argument is verbatim over $\mathbb C$) identify the corresponding reductions of \(S\) with the fifth Veronese subrings of the reductions of $H_{\mathbb C}$, that is,
\[
S\simeq H_{\mathbb C}^{(5)}.
\]
Set \(f=A^2-XY(X^2-Y)\). Then
\[
\partial_A f=2A,\qquad
\partial_X f=Y(Y-3X^2),\qquad
\partial_Y f=X(-X^2+2Y).
\]
These three derivatives vanish together with \(f\) only at the origin; if \(Y=0\), then \(X^3=0\), and if \(Y=3X^2\), then \(5X^3=0\). Thus \(H_{\mathbb C}\) is a hypersurface whose singular locus has codimension \(2\), so it is normal by Serre's criterion. Moreover \(f\) is irreducible, since \(XY(X^2-Y)\) is not a square in
\(\mathbb C[X,Y]\). Hence \(H_{\mathbb C}\) is a domain.

By \cite[Proposition~4.3, Remark 4.4]{Sin99}, the reductions
\[
H_p:=H_{\mathbb Z}\otimes_{\mathbb Z}\mathbb F_p
\]
are \(F\)-regular for every prime \(p>2\). Hence \(H_{\mathbb C}\) is of strongly \(F\)-regular type. By \cite[Corollary~3.4]{Tak04}, \(\Spec H_{\mathbb C}\) is klt. The inclusion
\[
H_{\mathbb C}^{(5)}\subset H_{\mathbb C}
\]
is finite and split by projection onto the homogeneous summands whose degrees are divisible by \(5\). Thus
\[
\Spec H_{\mathbb C}\to \Spec S
\]
is pure. By \cite[Theorem~1.1]{Zhu24}, a pure image of an affine klt type variety is of klt type. Therefore \(\Spec S\) is of klt type, so
\[
X_0=\Spec(\mathcal R/s\mathcal R)
\simeq
\Spec S\times \mathbb A^1_E
\]
is of klt type.

We next verify flatness and normality. The ring \(R\) is positively graded by
\[
\deg A=5,\qquad \deg B=\deg C=10,\qquad \deg D=20,\qquad \deg T=2.
\]
Singh's determinantal calculation \cite[Remark~4.1]{Sin99} shows that \(R\) is Cohen--Macaulay of dimension \(3\) and that \(T\) is a nonzerodivisor (stated there over an algebraically closed field $K$ of characteristic $p>2$, but the argument is verbatim over $\mathbb C$). We prove that \(R\) is a domain. If \(0\neq x\in R\), then since \(T\) has positive degree and every element of \(R\) has only
finitely many homogeneous components, for every nonzero \(x\in R\)
there exists a maximal \(q\geq 0\) such that \(x\in T^qR\). If \(xy=0\) with \(x,y\neq 0\), choose
\[
x=T^rx',\qquad y=T^sy',
\qquad x'\notin TR,\qquad y'\notin TR.
\]
Since \(T\) is a nonzerodivisor, \(x'y'=0\). Reducing modulo \(T\) gives
\[
\overline{x'}\,\overline{y'}=0
\]
in \(S=R/TR\), which is a domain. This contradicts the choice of \(x'\) and \(y'\). Thus \(R\) is a domain.

The same Jacobian computation as in \cite[Proposition~6.2]{Sin99} (stated there over an algebraically closed field $K$ of characteristic $p>2$, but the argument is verbatim over $\mathbb C$), applied over \(\mathbb C\), gives
\[
\Sing R\subset V(A,B,D,C(C+T^5)).
\]
The right-hand side is the union of
\[
V(A,B,C,D)\qquad\text{and}\qquad V(A,B,D,C+T^5),
\]
and each component has dimension \(1\). Since \(\dim R=3\), the singular locus has codimension at least \(2\). Thus \(R\) satisfies \(R_1\). Since \(R\) is Cohen--Macaulay, it satisfies \(S_2\). Hence \(R\) is normal.

Moreover
\[
\mathcal R\simeq R[s,E]/(T-sE).
\]
The ring \(R[s,E]\) is a Cohen--Macaulay domain, and \(T-sE\) is a nonzerodivisor, so \(\mathcal R\) is Cohen--Macaulay. Multiplication by \(s\) on \(\mathcal R\) is injective: if
\[
sx=(T-sE)g
\]
in \(R[s,E]\), then reducing modulo \(s\) gives \(T\overline g=0\) in \(R[E]\), hence \(\overline g=0\); thus \(g=sg'\), and then \(x\in (T-sE)\). Since
\[
\mathcal R_s\simeq R[s,s^{-1}]
\]
is a domain and \(s\) is a nonzerodivisor on \(\mathcal R\), the ring \(\mathcal R\) is a domain.

Let \(P\subset \mathcal R\) be a prime of height $1$. If \(s\notin P\), then \(\mathcal R_P\) is a height-one localization of the normal ring \(R[s,s^{-1}]\). If \(E\notin P\), then \(\mathcal R_P\) is a height-one localization of
\[
\mathcal R_E\simeq R[E,E^{-1}].
\]
In both cases \(\mathcal R_P\) is regular. Finally, \(s,E\) cannot both lie in \(P\), because
\[
\mathcal R/s\mathcal R\simeq S[E]
\]
and \(E\) is a nonzerodivisor on \(S[E]\), so \((s,E)\) is an \(\mathcal R\)-regular sequence. Hence \(\mathcal R\) satisfies \(R_1\). Since \(\mathcal R\) is Cohen--Macaulay, Serre's criterion implies that \(\mathcal R\) is normal.

The morphism \(\pi\) is flat. The map \(\mathbb C[T]\to R\) is injective because it admits the retraction
\[
R\to \mathbb C[T],
\qquad
A,B,C,D\mapsto 0.
\]
Since \(R\) is a domain, it is torsion-free over the PID \(\mathbb C[T]\), hence flat. By base change,
\[
\mathcal R=R\otimes_{\mathbb C[T]}\mathbb C[s,E]
\]
is flat over \(\mathbb C[s,E]\), and therefore over \(\mathbb C[s]\).

The closed fibers are identified as follows. For \(s=0\),
\[
\mathcal R/s\mathcal R\simeq R[E]/(T)\simeq S[E],
\]
so \(X_0\simeq \Spec S\times \mathbb A^1_E\), already shown to be of klt type. If \(a\in\mathbb C^*\), then the change of variables \(T=aE\) gives
\[
\mathcal R/(s-a)
\simeq
\mathbb C[A,B,C,D,E]/I_{aE}
\simeq
R.
\]
Thus \(X_a\simeq \Spec R\) for every \(a\neq 0\). In particular all closed fibers are normal.

It remains to see that the nonzero fibers are not of klt type. Let
\[
R_{\mathbb Z}=\mathbb Z[A,B,C,D,T]/I_{T,\mathbb Z},
\qquad
R_p=R_{\mathbb Z}\otimes_{\mathbb Z}\mathbb F_p,
\]
where \(I_{T,\mathbb Z}\) is generated by the same three $2\times 2$ minors over \(\mathbb Z\). By \cite[Proposition~4.5]{Sin99}, for \(m=5\) and \(n=2\), the ring \(R_p\) is not \(F\)-pure for every prime \(p\nmid 10\). 

Suppose that \(\Spec R\) were of klt type. Then there exists a
\(\mathbb Q\)-divisor \(\Delta\geq 0\) such that \((\Spec R,\Delta)\) is klt.
By \cite[Corollary~3.4]{Tak04}, after spreading out, for closed points
\(\mu\) in a dense open subset of \(\Spec A_0\), the corresponding
reduction is strongly \(F\)-regular as a pair; in particular, the underlying
ring is \(F\)-pure. This is not possible because the image in \(\Spec \mathbb Z\) of a dense open subset of \(\Spec A_0\) contains a dense set of primes, but $R_p$ is not $F$-pure for any prime $p\nmid 10$.

Consequently, \(X_0\) is of klt type, while \(X_s\) is not of klt type for every \(s\in\mathbb C^*\). Since every Zariski-open neighborhood of \(0\) in \(\mathbb A^1_{\mathbb C}\) contains a nonzero closed point, the klt-type fiber locus is not open. This proves the theorem.
\end{proof}

\begin{rem}
Junpeng Jiao informed that our result is more or less well-known by experts (although the author cannot find any direct reference, particularly after \cite{HJ26}), and the more natural question to ask is whether klt is a constructible condition or not. As a counterpart, it is interesting to ask whether being Fano type is a constructible condition or not. 
\end{rem}

\begin{rem}
The counterexample in this paper was obtained with the assistance of generative AI. The general idea for the construction was suggested by ChatGPT Pro 5.5; the explicit example was then found and proved by the Rethlas system. See \cite{Ju+26} for a detailed introduction to the Rethlas system.

Due to the limitation of generative AI, it is possible that we have missed some related references in the literature, and we welcome any comments from experts.
\end{rem}

\subsection*{Acknowledgements}
The author was partially supported by the National Key R\&D Program of China \#\allowbreak 2024YFA1014400. The author would like to thank the Rethlas team, namely Haocheng Ju, Jiedong Jiang, Shurui Liu, Guoxiong Gao, Yuefeng Wang, Zeming Sun, Bin Wu, Liang Xiao, and Bin Dong, for their contributions to the development of Rethlas and its customized version used for the problem studied in this paper. The author would like to thank Kaiyuan Gu, Ruicheng Hu, and Sheng Qin for assistance with the verification of an earlier blueprint of this paper. The author would like to thank Ruochuan Liu and Gang Tian for constant support and encouragement. The author would like to thank Junpeng Jiao for useful discussions.

\end{document}